\documentclass[11pt,a4paper]{article}
\usepackage[T1]{fontenc}
\usepackage{lmodern}
\usepackage{times}
 \usepackage[pagebackref]{hyperref}
\hypersetup{backref, pagebackref, colorlinks=true}
\usepackage[margin=28mm]{geometry}
\usepackage{amsmath,amssymb,amsthm}
\usepackage{microtype}
 
\numberwithin{equation}{section}
\newtheorem{theorem}{Theorem}[section]
\newtheorem{lemma}[theorem]{Lemma}

\newtheorem{corollary}[theorem]{Corollary}
\newtheorem{definition}[theorem]{Definition}

\theoremstyle{remark}

\newcommand{\Z}{\mathbb Z}

\newcommand{\C}{\mathbb C}
\newcommand{\F}{\mathbb F}

\newcommand{\Hh}{\mathbb H}

\newcommand{\Mmod}{\mathcal M}

\newcommand{\red}[1]{\overline{#1}}
\title{Nonexistence of Constant Congruences\texorpdfstring{\\}{ }for the Cumulative Partition Function}
\author{Kathy Q. Ji}
\date{Center for Applied Mathematics\\[2pt]
Tianjin University\\[2pt]
Tianjin 300072, P.R. China\\[5pt]
Email: kathyji@tju.edu.cn
}

\begin{document}
\maketitle
\vspace{-2.5em}

\begin{abstract}
Let $p(n)$ be the ordinary partition function and let
$c(n)=\sum_{k=0}^{n-1}p(k)$, with $c(0)=0$. We prove that, for every $a\geq1$, $b\geq0$,
and $m\geq2$, the sequence $c(an+b)$ is not eventually constant modulo $m$, thereby answering a question posed by Chen.
The proof combines the algebraicity over $\mathbb C(j)$
of modular functions for finite-index subgroups of
$\mathrm{SL}_2(\mathbb Z)$ with Voloch's theorem on the
transcendence of elliptic modular functions in positive
characteristic.
\end{abstract}

\section{Introduction and the main result}

The ordinary partition function $p(n)$ counts the ways of
expressing $n$ as an unordered sum of positive integers.
We adopt the conventions $p(0)=1$ and $p(n)=0$ for $n<0$.
Its generating function is
\begin{equation}\label{eq:partition}
 P(x):=\sum_{n\geq0}p(n)x^n
 =\prod_{n\geq1}(1-x^n)^{-1}.
\end{equation}
For background on partition theory, see
Andrews~\cite{Andrews1998}.

We consider the cumulative sum
\begin{equation}\label{eq:cumulative}
 c(n):=\sum_{k=0}^{n-1}p(k)
 \qquad(n\geq0),
\end{equation}
where an empty sum is understood to be zero, so that $c(0)=0$.
Thus $c(n)$ counts the partitions of all nonnegative integers
strictly less than $n$, including the empty partition when $n\geq1$.
Cumulative sums of the partition function were studied by
Gupta~\cite{Gupta1946}.
Following the terminology of Chen~\cite{Chen2026},
we refer to $c(n)$ as the \emph{cumulative partition function}.

The first few values are
\[
\begin{array}{c|rrrrrrrrrrr}
n    &0&1&2&3&4&5&6&7&8&9&10\\ \hline
p(n) &1&1&2&3&5&7&11&15&22&30&42\\
c(n) &0&1&2&4&7&12&19&30&45&67&97
\end{array}
\]
With the normalization in \eqref{eq:cumulative},
the generating function of $c(n)$ is
\begin{equation}\label{eq:cumulative-gf}
 C(x):=\sum_{n\geq0}c(n)x^n
 =\sum_{k\geq0}p(k)\frac{x^{k+1}}{1-x}
 =\frac{xP(x)}{1-x}.
\end{equation}

Chen~\cite{Chen2026} asked whether the cumulative partition
function can satisfy a congruence with a constant residue
along an arithmetic progression.
Our main result shows that no such congruence can hold,
even after finitely many initial terms are omitted.

\begin{theorem}\label{thm:main}
Let $a\geq1$, $b\geq0$, and $m\geq2$ be integers.
There do not exist integers $d$ and $n_0\geq0$ such that
\begin{equation}\label{eq:forbidden}
 c(an+b)\equiv d\pmod m
 \qquad\text{for every }n\geq n_0.
\end{equation}
\end{theorem}

The theorem has the following immediate consequences.

\begin{corollary}\label{cor:parity}
For all integers $a\geq1$ and $b\geq0$, there are infinitely
many integers $n\geq0$ for which $c(an+b)$ is even and
infinitely many for which it is odd.
\end{corollary}

\begin{proof}
If either parity occurred only finitely often, then
$c(an+b)$ would have the other parity for all sufficiently
large $n$. The sequence would therefore be eventually
constant modulo $2$, contrary to Theorem~\ref{thm:main}.
\end{proof}

\begin{corollary}\label{cor:periodicity}
For all integers $a\geq1$, $b\geq0$, and $m\geq2$,
the sequence $(c(an+b))_{n\geq0}$ is not eventually
periodic modulo $m$.
\end{corollary}

\begin{proof}
Suppose, to the contrary, that the sequence is eventually
periodic modulo $m$.
Then there exist integers $t\geq1$ and $N\geq0$ such that
\[
 c(a(n+t)+b)\equiv c(an+b)\pmod m
 \qquad\text{for every }n\geq N.
\]
Iterating this congruence gives
\[
 c(atk+aN+b)\equiv c(aN+b)\pmod m
 \qquad\text{for every }k\geq0.
\]
This is a constant congruence along an arithmetic progression
with step $at\geq1$ and initial term $aN+b\geq0$,
contradicting Theorem~\ref{thm:main}.
\end{proof}

 The proof of our main result relies on Voloch's theorem on the
transcendence of elliptic modular functions in characteristic $p$.
In Section~2, we review the necessary background on modular functions
for finite-index subgroups of $\mathrm{SL}_2(\mathbb Z)$, together
with basic facts about algebraic extensions and transcendence over
fields. Section~3 establishes the auxiliary results needed for the
proof of Theorem~\ref{thm:main}, and Section~4 is devoted to its proof.

\section{Preliminaries}\label{sec:algebraicity}

In this section, we review the necessary background on modular functions
for finite-index subgroups of $\mathrm{SL}_2(\mathbb Z)$, together with
basic facts about algebraic extensions and transcendence over fields.

Throughout,  $\ell$ denotes a prime.
All Laurent series below are understood formally; no convergence
condition is imposed. We use the following notation.

\begin{enumerate}
\item  $\mathbb C[X]$ is the ring of polynomials in an indeterminate
$X$ with complex coefficients. Replacing $\mathbb Z$ by a field $k$ gives the the ring of polynomials  $k[X]$. In particular, we shall use $\mathbb C[X]$
and $\mathbb F_\ell[X]$, where $\mathbb F_\ell=\mathbb Z/\ell\mathbb Z$. 

\item
For independent indeterminates $U$ and $V$, the ring
$\mathbb Z[U,V]$ consists of all polynomials
\[
A(U,V)=\sum_{i,j\geq0}a_{ij}U^iV^j,
\qquad a_{ij}\in\mathbb Z,
\]
with only finitely many nonzero coefficients.
The exponents of $U$ and $V$ are nonnegative integers.
A nonzero polynomial in $\mathbb Z[U,V]$ is called
\emph{primitive} if the greatest common divisor of its
coefficients is $1$.

\item  
The ring of formal Laurent series with integer coefficients is
\[
\mathbb Z((q))
=\left\{
\sum_{n=n_0}^{\infty}a_nq^n:
n_0\in\mathbb Z,\ a_n\in\mathbb Z
\right\}.
\]
Thus an element of $\mathbb Z((q))$ may contain negative powers
of $q$, but only finitely many of them. Addition and multiplication
are defined by the usual formal series operations.
This is a ring, not a field; 
Replacing $\mathbb Z$ by a field $k$ gives the formal Laurent series
field $k((q))$. In particular, we shall use $\mathbb C((q))$
and $\mathbb F_\ell((q))$.

\item
Let $V$ be an indeterminate. The field of rational functions
in $V$ over $\mathbb C$ is the fraction field of the polynomial
ring $\mathbb C[V]$, namely,
\[
\mathbb C(V)
=\operatorname{Frac}(\mathbb C[V])
=\left\{
\frac{P(V)}{Q(V)}:
P(V),Q(V)\in\mathbb C[V],\ Q(V)\neq0
\right\}.
\]
Here $\mathbb C[V]$ consists of all finite sums
$a_0+a_1V+\cdots+a_dV^d$ with $a_i\in\mathbb C$.

\item
Let $\mathbb F_\ell=\mathbb Z/\ell\mathbb Z$.
For $a\in\mathbb Z$, write $\overline a$ for its image
in $\mathbb F_\ell$. Coefficientwise reduction modulo $\ell$
defines a ring homomorphism
\[
\mathbb Z((q))\longrightarrow\mathbb F_\ell((q)),
\qquad
\sum_{n=n_0}^{\infty}a_nq^n
\longmapsto
\sum_{n=n_0}^{\infty}\overline{a_n}q^n.
\]

\end{enumerate}
We first recall some basic facts about modular functions for
finite-index subgroups of $\mathrm{SL}_2(\mathbb Z)$; see
\cite{Apostol1990,Milne2017,Silverman2009} for further details.
The \emph{modular group} is the group of $2\times2$ matrices
with integer entries and determinant $1$:
\[
\Gamma:=\mathrm{SL}_2(\mathbb Z)
=
\left\{
\begin{pmatrix}
a & b\\
c & d
\end{pmatrix}
: a,b,c,d\in\mathbb Z,\ ad-bc=1
\right\}.
\]
We denote the upper half-plane by
\[
\Hh:=\{\tau\in\mathbb C:\operatorname{Im}\tau>0\}.
\]

Let $\gamma=\begin{pmatrix}a &b\\ c &d\end{pmatrix}\in\Gamma$ act on $\tau\in \Hh$ by the linear fractional transformation
\begin{align*}
    \gamma\tau = \frac{a\tau+b}{c\tau+d}, \qquad \text{and} \qquad
    \gamma\infty =\lim_{\tau\rightarrow \infty} \gamma\tau.
\end{align*}

\begin{definition}\label{def:modular-function}
Let $\Gamma'$ be a finite-index subgroup of
$\mathrm{SL}_2(\mathbb Z)$.
A \emph{modular function for $\Gamma'$} is a meromorphic
function $f$ on $\Hh$ satisfying the following two conditions:
\begin{enumerate}
 \item
 $f(\gamma\tau)=f(\tau)$ for every $\gamma\in\Gamma'$.

 \item    for  $\gamma \in \mathrm{SL}_2(\mathbb Z)$, $f(\gamma\tau)$ has a Fourier expansion
of the form
\begin{align*}
 f(\gamma\tau)=\sum_{n=n_\gamma}^\infty a_\gamma(n)q_{w_\gamma}^n, \quad q_{w_\gamma}=e^{2\pi i\tau/w_{\gamma}},
\end{align*}
where $a_\gamma(n_\gamma)\neq 0$ 
and $w_\gamma$ is the minimal positive integer $h$ such that
\begin{align*}
\begin{pmatrix}1 &h\\ 0 &1\end{pmatrix}\in\gamma^{-1}\Gamma'\gamma.
\end{align*}

\end{enumerate}

We denote the set of these functions by $\Mmod(\Gamma')$.
\end{definition}

Note that Condition~(2) in Definition \ref{def:modular-function} means that $f(\tau)$ is meromorphic at every cusp of $\Gamma'$.

 For $\tau\in\mathbb H$, put $q=e^{2\pi i\tau}$.
Recall the Dedekind eta function, the modular discriminant,
and the Eisenstein series of weight $4$:
\begin{align}
 \eta(\tau)
 &=q^{1/24}\prod_{n\geq1}(1-q^n),
 \label{eq:eta}\\
 \Delta(\tau)
 &=\eta(\tau)^{24}
 =q\prod_{n\geq1}(1-q^n)^{24},
 \label{eq:delta}\\
 E_4(\tau)
 &=1+240\sum_{n\geq1}\sigma_3(n)q^n,
 \label{eq:E4}
\end{align}
where
\[
 \sigma_3(n)=\sum_{\substack{d\mid n\\d>0}}d^3.
\]

The discriminant $\Delta$ is a cusp form of weight $12$
for $\mathrm{SL}_2(\mathbb Z)$; see
\cite[Proposition~4.16(c), pp.~54--55,
and Theorem~4.21, p.~57]{Milne2017}.
In particular, it is holomorphic on $\mathbb H$ and satisfies
\begin{equation}\label{eq:drp-transformation}
 \Delta(\gamma\tau)=(c\tau+d)^{12}\Delta(\tau)
 \qquad
 \text{for every }
 \gamma=
 \begin{pmatrix}a&b\\c&d\end{pmatrix}\in\Gamma.
\end{equation}
Moreover, $\Delta(\tau)\neq0$ for every $\tau\in\mathbb H$.

The modular $j$-invariant is defined by
\begin{equation}\label{eq:j}
 j(\tau)=\frac{E_4(\tau)^3}{\Delta(\tau)}.
\end{equation}
Since $E_4$ and $\Delta$ have weights $4$ and $12$,
respectively, their transformation laws give
\[
 j(\gamma\tau)=j(\tau)
 \qquad\text{for every }\gamma\in\Gamma.
\]
The nonvanishing of $\Delta$ on $\mathbb H$ also shows
that $j$ is holomorphic there.
At the cusp $\infty$, it has a simple pole.
Thus $j$ is a modular function for
$\mathrm{SL}_2(\mathbb Z)$; see
\cite[Chapter~2, \S2.4, p.~34]{Apostol1990},
where the normalization is $J=j/1728$.

More explicitly, the Fourier expansion of $j$ is
\begin{equation}\label{eq:j-expansion}
 j(q):=j(\tau)
 =q^{-1}+744+196884q+21493760q^2+\cdots
 \in\mathbb Z((q));
\end{equation}
see \cite[Theorem~1.20 and the subsequent coefficient table,
pp.~21--22]{Apostol1990}, where $j=12^3J$, or
\cite[Appendix~C, \S14, p.~444]{Silverman2009}.

We next recall some terminology from field theory.
Let $E\subseteq F$ be fields, and let $z\in F$.
The element $z$ is called \emph{algebraic over $E$}
if there exists a nonzero polynomial $B(T)\in E[T]$
such that $B(z)=0$.
If no such polynomial exists, then $z$ is called
\emph{transcendental over $E$}.

We write $\mathbb C(j(q))$ for the smallest subfield of
$\mathbb C((q))$ containing $\mathbb C$ and $j(q)$.
Explicitly,
\[
 \mathbb C(j(q))
 =
 \left\{
 \frac{P(j(q))}{Q(j(q))}:
 P(X),Q(X)\in\mathbb C[X],\ Q(X)\neq0
 \right\},
\]
where $\mathbb C[X]$ denotes the ring of polynomials
in an indeterminate $X$ with complex coefficients.
Thus the elements of $\mathbb C(j(q))$ are rational
functions of $j(q)$ with complex coefficients.
Here we identify
$\mathbb C(j(\tau))$ with $\mathbb C(j(q))$
and abbreviate this field to $\mathbb C(j)$.

We shall use the following standard algebraicity theorem;
see \cite[Chapter~V, Section~1, Theorem~1, p.~105]
{Schoeneberg1974}.

\begin{theorem}\label{thm:algebraic}
Let $\Gamma'$ be a finite-index subgroup of
$\mathrm{SL}_2(\mathbb Z)$, and let $\Mmod(\Gamma')$
denote the field of meromorphic modular functions for
$\Gamma'$, meromorphic also at every cusp.
Then every $f\in\Mmod(\Gamma')$ is algebraic over
$\mathbb C(j)$.
\end{theorem}

We also recall the following descent lemma for algebraic
dependence; see
\cite[Lemma~7.2, p.~533]{AdamczewskiBellDelaygue2019}.

\begin{lemma}\label{lem:algebraic-dependence-descent}
Let $K$ be a subfield of $\mathbb C$, and let $n\geq1$.
Suppose that
\[
 f_1(x),\ldots,f_n(x)\in K[[x]]
\]
are algebraically dependent over $\mathbb C(x)$.
Then they are algebraically dependent over $K(x)$.
\end{lemma}

The following lemma concerns descent to rational
coefficients and reduction modulo a prime.
It will be used in the proof of Theorem~\ref{thm:main}.

\begin{lemma}[Descent and reduction]\label{lem:descent}
Suppose that $h(q)\in\mathbb Z((q))$ is algebraic over
$\mathbb C(j(q))$.
Then there exists a nonzero polynomial
$A(U,V)\in\mathbb Z[U,V]$, whose coefficients have
greatest common divisor $1$, such that
\begin{equation}\label{eq:primitive}
 A(h(q),j(q))=0.
\end{equation}
Moreover, for every prime $\ell$, the coefficientwise
reduction $\overline h(q)$ is algebraic over
$\mathbb F_\ell(\overline j(q))$.
\end{lemma}

\begin{proof}
By the assumed algebraicity, there is a relation
\[
 \sum_{i=0}^{d}r_i(\red{j}(q))h(q)^i=0,
 \qquad r_i(V)\in\mathbb C(V),
\]
Since $h(q)\in\mathbb Z((q))$, by Lemma \ref{lem:algebraic-dependence-descent}, we deduce that there is a relation
\[
 \sum_{i=0}^{d}q_i(\red{j}(q))h(q)^i=0,
 \qquad q_i(V)\in\mathbb Q(V),
\]

  We have proved
the existence of a nonzero polynomial
\[
 B(U,V)\in\mathbb Q[U,V],\qquad B(h(q),\red{j}(q))=0.
\]

Choose a positive integer $m$ clearing the denominators of
all coefficients of $B$, and put $C(U,V)=mB(U,V)$.
Then $C\in\mathbb Z[U,V]$ is nonzero and
$C(h(q),\red{j}(q))=0$.
Let $c>0$ be the greatest common divisor of its coefficients,
and define
\[
 A(U,V)=\frac{C(U,V)}{c}.
\]
Every coefficient of $C$ is divisible by $c$, so
$A\in\mathbb Z[U,V]$. It is nonzero, its coefficients have
greatest common divisor $1$, and it satisfies
\eqref{eq:primitive}. Such a polynomial is called
\emph{primitive}.

Fix a prime $\ell$. Reduction of coefficients defines a ring
homomorphism
\[
 \mathbb Z((q))\longrightarrow\mathbb F_\ell((q)),
 \qquad
 \sum_{n\geq n_0}a_nq^n
 \longmapsto
 \sum_{n\geq n_0}(a_n\bmod\ell)q^n.
\]
This respects multiplication because the coefficient of each
power of $q$ in a product of Laurent series is a finite sum.
Reducing~\eqref{eq:primitive} therefore gives
\begin{equation}\label{eq:desred-reduced}
 \overline A\bigl(\overline h(q),\red{j}(q)\bigr)=0,
\end{equation}
where $\overline A(U,V)\in\mathbb F_\ell[U,V]$ is obtained by
where $\overline A(U,V)\in\mathbb F_\ell[U,V]$ is obtained by
reducing every integer coefficient of $A(U,V)$. Note that 
$\overline A(U,V)\neq0$,  otherwise every coefficient of $A(U,V)$
would be divisible by $\ell$, contradicting primitivity.

We proceed to show that  $\overline A(U, \red{j}(q))$  is a nonzero polynomial. 
Write
\[
 \overline A(U,V)=\sum_{i=0}^{r}a_i(V)U^i,
 \qquad a_i(V)\in\mathbb F_\ell[V].
\]
Since $\overline A\neq0$, then there is at least one $a_i(V)$ is nonzero.
Assume that  
\[
 a_i(V)=a_i^eV^e+\cdots+a_i^0 \in\mathbb F_\ell[V],
\]
the term $a_e( \red{j}(q))^e$ has leading term $a_eq^{-e}$.
Every lower-degree term $a_i(  \red{j}(q))^i$, with $i<e$, has
no power of $q$ as low as $q^{-e}$. Hence $
a_i( \red{j}(q))\neq0.$ Every such nonzero coefficient
remains nonzero after evaluating at $  \red{j}(q)$. Thus
\[
 P_\ell(U):=\overline A(U,  \red{j}(q))
 \in\mathbb F_\ell(  \red{j}(q))[U]
\]
is a nonzero polynomial.
Equation~\eqref{eq:desred-reduced} says precisely that
$P_\ell(\overline h(q))=0$.
Moreover, $P_\ell$ has positive degree in $U$, since a nonzero
constant polynomial cannot vanish.
Therefore $\overline h(q)$ is algebraic over
$\mathbb F_\ell(\red{j}(q))$. 
\end{proof}

 We record a standard fact from field theory; see
\cite[Proposition~1.30, pp.~17--18]{Milne2022} and
\cite[Section~13.2, Theorem~17, p.~526]{DummitFoote2004}.

\begin{lemma}\label{lem:finite-algebraic-adjunction}
Let $L$ be a field, let $\Omega$ be an extension field
of $L$, and let $u_0,\ldots,u_{M-1}\in\Omega$,
where $M\geq1$.
If each $u_t$ is algebraic over $L$, then
\[
 E=L(u_0,\ldots,u_{M-1})
\]
is a finite extension of $L$.
In particular, every element of $E$ is algebraic over $L$.
\end{lemma}

We conclude this section by stating Voloch's theorem,
which plays a central role in the proof of our main result.
To formulate it, we first introduce the formal power series
\begin{align}
 A_4(q)
 &=-5\sum_{n\geq1}\frac{n^3q^n}{1-q^n},
 \label{eq:A4}\\
 A_6(q)
 &=-\sum_{n\geq1}
   \frac{5n^3+7n^5}{12}\frac{q^n}{1-q^n}.
 \label{eq:A6}
\end{align}
Both series belong to $\mathbb Z[[q]]$.

Consider the Tate curve over $\mathbb Q((q))$ given by
the Weierstrass equation
\begin{equation}\label{tatecurve}
 E_q:\quad y^2+xy=x^3+A_4(q)x+A_6(q).
\end{equation}
The discriminant of this equation is the normalized
modular discriminant $\Delta(q)$, and its $j$-invariant
is the classical series $j(q)$; see
\cite[Appendix~C, Section~14, p.~444]{Silverman2009}.

We now express these invariants in terms of $A_4(q)$
and $A_6(q)$.
For the equation \eqref{tatecurve}, the Weierstrass
coefficients are
\[
 a_1=1,\qquad a_2=a_3=0,\qquad
 a_4=A_4(q),\qquad a_6=A_6(q).
\]
Writing $A_4=A_4(q)$ and $A_6=A_6(q)$ for brevity,
the standard invariant formulas
\cite[Chapter~III, Section~1, p.~42]{Silverman2009}
give
\[
 b_2=1,\qquad
 b_4=2A_4,\qquad
 b_6=4A_6,\qquad
 b_8=A_6-A_4^2,
\]
and
\[
 c_4=b_2^2-24b_4=1-48A_4.
\]
Substituting these expressions into
\[
 \Delta=-b_2^2b_8-8b_4^3-27b_6^2+9b_2b_4b_6,
 \qquad
 j=\frac{c_4^3}{\Delta},
\]
we obtain
\begin{align}
 \Delta(q)
 &=A_4(q)^2-A_6(q)-64A_4(q)^3
   -432A_6(q)^2+72A_4(q)A_6(q),
 \label{eq:invariant-delta}\\
 j(q)
 &=\frac{(1-48A_4(q))^3}{\Delta(q)}.
 \label{eq:invariant-j}
\end{align}
These identities hold in $\mathbb Z((q))$ and remain
valid after coefficientwise reduction modulo any prime.
Indeed, $\Delta(q)=q+O(q^2)$ is a unit in
$\mathbb Z((q))$, and its reduction has leading term
$q$, so it remains nonzero in $\mathbb F_\ell((q))$
for every prime $\ell$.

We can now state Voloch's transcendence theorem.

\begin{theorem}[Voloch]\label{thm:voloch}
Let $\ell$ be a prime, and let
$k=\overline{\mathbb F}_\ell$ be an algebraic closure
of $\mathbb F_\ell=\mathbb Z/\ell\mathbb Z$.
Then the formal parameter $q$ is transcendental over
\[
 k\bigl(\red{A_4}(q),\red{A_6}(q)\bigr)
 \subseteq k((q)),
\]
where the bars denote coefficientwise reduction
modulo $\ell$.
\end{theorem}

This is Theorem~A of Voloch~\cite{Voloch1996}.
An alternative proof using automata theory was given
by Allouche and Thakur~\cite{AlloucheThakur1999}.

\section{Auxiliary results}

In this section, we establish several auxiliary results
needed for the proof of Theorem~\ref{thm:main}.

\begin{theorem}\label{prop:delta-ratio}
For $u\in\mathbb Q_{>0}$ and $v\in\mathbb Q$, the ratio
\begin{equation}\label{eq:ratio}
 g(\tau):=\frac{\Delta(\tau)}{\Delta(u\tau+v)}
\end{equation}
is algebraic over $\mathbb C(j)$. More precisely, there are an integer
$d\geq1$ and rational functions
$R_0(X),\ldots,R_{d-1}(X)\in\mathbb C(X)$ such that
\[
 g(\tau)^d+
 R_{d-1}(j(\tau))g(\tau)^{d-1}
 +\cdots+R_0(j(\tau))=0.
\]
\end{theorem}

\begin{proof}
By Theorem \ref{thm:algebraic}, it suffices to show that   $ g(\tau)$ is a modular function for a finite-index subgroup of $\mathrm{SL}_2(\mathbb Z)$. 

Let
\[
 \alpha=\begin{pmatrix}u&v\\0&1\end{pmatrix},
 \qquad
 \Gamma_\alpha=\Gamma\cap\alpha^{-1}\Gamma\alpha.
\]

We aim to show that $ g(\tau)$ is a modular function for $ \Gamma_\alpha$. 
We first verify that $g$ is holomorphic and nowhere zero on
the upper half-plane $\mathbb H$.
 Since $\Delta (\tau)$ is holomorphic and
nowhere zero on $\mathbb H$. Note that $u>0$ and $v\in\mathbb R$, for every
$\tau\in\mathbb H$, $\operatorname{Im}(u\tau+v)
=
u\,\operatorname{Im}\tau>0.$ 
Therefore $u\tau+v\in\mathbb H$,   consequently, $\Delta(u\tau+v)$ is
holomorphic and nowhere zero on $\mathbb H$, so its
reciprocal is holomorphic there.
It follows that $g(\tau)$  
is holomorphic on $\mathbb H$.  

We next show that $\Gamma_\alpha$ is a finite index subgroup of $\Gamma$.  Choose integers $A,D>0$ and $B\in\mathbb Z$ such that
$u=A/D$ and $v=B/D$. Set 
\[
 M=\begin{pmatrix}A&B\\0&D\end{pmatrix}=D\alpha,
 \qquad N=AD=\det M.
\]
Let
\[
 \Gamma(N)=\{\gamma\in G:\gamma\equiv I\pmod N\}
\]
be   the congruence subgroup of $\Gamma$. It is known that  \(\Gamma(N)\) is 
a finite-index subgroup  of $\Gamma$.  We proceed to show that 
$\Gamma(N)\subseteq\Gamma_\alpha$.  
Let $\gamma=I+NX\in\Gamma(N)$, where $X$ is an
integer matrix. Since
$M^{-1}=\operatorname{adj}(M)/N$, we have
\begin{align*}
 \alpha\gamma\alpha^{-1} &=(\frac{1}{D}M)\gamma (DM^{-1})=M\gamma M^{-1}\\
 &=M(I+NX) M^{-1}\\
 &=I+NMxM^{-1}\\
 &=I+NMX\frac{\operatorname{adj}(M)}{N}\\
 &=I+MX\operatorname{adj}(M).
\end{align*}
The matrix on the right has integer entries and determinant
$\det(\gamma)=1$. Hence $\alpha\gamma\alpha^{-1} \in \Gamma$. It yields that $\Gamma(N)\subseteq\Gamma_\alpha$, and so $\Gamma_\alpha$ is a finite index subgroup
of $\Gamma$.

(3) We will show that $g(\tau)$ is $\Gamma_\alpha$-invariant. 
Take
\[
 \gamma=\begin{pmatrix}a&b\\c&d\end{pmatrix}
 \in\Gamma_\alpha,
 \qquad
 \delta=\alpha\gamma\alpha^{-1}=\begin{pmatrix}a_\delta&b_\delta\\c_\delta&d_\delta\end{pmatrix}\in \Gamma.
\]
 
Comparing lower rows in $\delta\alpha=\alpha\gamma$ gives
\[
 c_\delta u=c,\qquad c_\delta v+d_\delta=d,
\]
and therefore
\[
 c_\delta(u\tau+v)+d_\delta=c\tau+d.
\]
Also $\alpha\gamma\tau=\delta(\alpha\tau)$.
Applying~\eqref{eq:drp-transformation} in both the numerator
and denominator, we obtain
\[
 \begin{aligned}
 g(\gamma\tau)
 &=\frac{\Delta(\gamma\tau)}
         {\Delta(\alpha\gamma\tau)}\\
 &=\frac{(c\tau+d)^{12}\Delta(\tau)}
         {(c_\delta(u\tau+v)+d_\delta)^{12}
          \Delta(u\tau+v)}\\
 &=g(\tau).
 \end{aligned}
\]
Thus $g$ is $\Gamma_\alpha$-invariant.

\medskip 

(4) It remains to verify that  $g(\tau)$ is 
meromorphic at  every cusp of $\gamma_\alpha$. 

Fix $\sigma=\begin{pmatrix}
a&b\\
c&d
\end{pmatrix}\in \Gamma$.  Since $\alpha\tau=u\tau+v$ with
$u\in\mathbb Q_{>0}$ and $v\in\mathbb Q$, we have
$\alpha\sigma\infty\in\mathbb Q\cup\{\infty\}$.
In this case, we may choose $\rho\in \Gamma$ such that
$\rho\infty=\alpha\sigma\infty.$ Assume that $\rho=\begin{pmatrix}
A&B\\
C&D
\end{pmatrix}$. Then $\beta=\rho^{-1}\alpha\sigma$ fixes $\infty$, so we may set 
\[
\beta=
\begin{pmatrix}
\lambda&\mu\\
0&\nu
\end{pmatrix},
\qquad
\lambda,\mu,\nu\in\mathbb Q.
\]
Since $\rho,\sigma\in\mathrm{SL}_2(\mathbb Z)$ and
$\det(\alpha)=u$, we have
\[
\lambda\nu
=\det(\beta)
=\det(\rho^{-1})\det(\alpha)\det(\sigma)
=u>0.
\]

Put
\[
s=\frac{\lambda}{\nu}>0,
\qquad
t=\frac{\mu}{\nu}\in\mathbb Q.
\]
Thus $\beta\tau=\frac{\lambda\tau+\mu}{\nu}=s\tau+t$.
 
The fractional linear transformation associated with $\beta$
is therefore
\[
\beta\tau=\frac{\lambda\tau+\mu}{\nu}=s\tau+t.
\]
Since $s>0$ and $t$ is real, we have
\[
\operatorname{Im}(\beta\tau)
=s\operatorname{Im}(\tau)>0
\qquad (\tau\in\mathbb H).
\]
Thus $\beta\tau\in\mathbb H$, so the transformation law for
$\Delta$ can be applied at the point $\beta\tau$.

We next compare the lower rows in the identity
$\rho\beta=\alpha\sigma$.
Direct matrix multiplication gives
\[
\rho\beta
=
\begin{pmatrix}
A&B\\
C&D
\end{pmatrix}
\begin{pmatrix}
\lambda&\mu\\
0&\nu
\end{pmatrix}
=
\begin{pmatrix}
A\lambda&A\mu+B\nu\\
C\lambda&C\mu+D\nu
\end{pmatrix},
\]
whereas
\[
\alpha\sigma
=
\begin{pmatrix}
u&v\\
0&1
\end{pmatrix}
\begin{pmatrix}
a&b\\
c&d
\end{pmatrix}
=
\begin{pmatrix}
ua+vc&ub+vd\\
c&d
\end{pmatrix}.
\]
Equality of these matrices implies
\[
C\lambda=c,
\qquad
C\mu+D\nu=d.
\]
Consequently,
\[
\begin{aligned}
C(\beta\tau)+D
&=C\left(\frac{\lambda\tau+\mu}{\nu}\right)+D\\
&=\frac{C\lambda\tau+C\mu+D\nu}{\nu}\\
&=\frac{c\tau+d}{\nu}.
\end{aligned}
\]

Recall that the weight-$12$ transformation law for $\Delta$
states that
\[
\Delta\left(\frac{az+b}{cz+d}\right)
=(cz+d)^{12}\Delta(z)
\]
for every
$\left(\begin{smallmatrix}a&b\\c&d\end{smallmatrix}\right)
\in\mathrm{SL}_2(\mathbb Z)$ and every $z\in\mathbb H$.
Applying this law to $\sigma$ at $z=\tau$ gives
\[
\Delta(\sigma\tau)
=(c\tau+d)^{12}\Delta(\tau).
\]
Applying it to $\rho$ at $z=\beta\tau$ gives
\[
\begin{aligned}
\Delta(\rho(\beta\tau))
&=(C(\beta\tau)+D)^{12}\Delta(\beta\tau)\\
&=\left(\frac{c\tau+d}{\nu}\right)^{12}
  \Delta(s\tau+t).
\end{aligned}
\]

By the definition
\[
g(z)=\frac{\Delta(z)}{\Delta(\alpha z)}
\]
and the identity $\alpha\sigma=\rho\beta$, we obtain
\begin{equation}\label{eq:drp-cusp-transform}
\begin{aligned}
g(\sigma\tau)
&=\frac{\Delta(\sigma\tau)}
        {\Delta(\alpha\sigma\tau)}=\frac{\Delta(\sigma\tau)}
        {\Delta(\rho(\beta\tau))}\\
&=\frac{(c\tau+d)^{12}\Delta(\tau)}
        {\left((c\tau+d)/\nu\right)^{12}
         \Delta(s\tau+t)}\\
&=\nu^{12}
  \frac{\Delta(\tau)}{\Delta(s\tau+t)}.
\end{aligned}
\end{equation}

Let
\[
T=\begin{pmatrix}1&1\\0&1\end{pmatrix},
\]
and let $h=w_\sigma$ be the least positive integer
such that
\[
T^h\in\sigma^{-1}\Gamma_\alpha\sigma.
\]
Such an integer exists because $\Gamma_\alpha$ has
finite index in $\Gamma$.

We claim that $hs$ is a positive integer.
Indeed, put
\[
\delta=\sigma T^h\sigma^{-1}\in\Gamma_\alpha.
\]
By the definition
$\Gamma_\alpha=\Gamma\cap\alpha^{-1}\Gamma\alpha$, we have
$\alpha\delta\alpha^{-1}\in \Gamma$.
Since $\rho\in \Gamma$, it follows that
\[
\beta T^h\beta^{-1}
=
\rho^{-1}\alpha\delta\alpha^{-1}\rho
\in \Gamma.
\]
Direct matrix multiplication gives
\[
\beta T^h\beta^{-1}
=
\begin{pmatrix}
1&h\lambda/\nu\\
0&1
\end{pmatrix}
=
\begin{pmatrix}
1&hs\\
0&1
\end{pmatrix}.
\]
Every matrix in $\Gamma$ has integer entries, so $hs\in\mathbb Z$.
Since $h>0$ and $s>0$, we may write
\[
m=hs\in\mathbb Z_{>0}.
\]

Now set
\[
q_h=e^{2\pi i\tau/h},
\qquad
\zeta=e^{2\pi it}.
\]
 
Substituting the product expansion of $\Delta$ into
\eqref{eq:drp-cusp-transform}, we obtain
\begin{equation}\label{eq:drp-laurent}
g(\sigma\tau)
=
\nu^{12}\zeta^{-1}q_h^{\,h-m}
\prod_{n=1}^{\infty}
\frac{(1-q_h^{hn})^{24}}
     {(1-\zeta^n q_h^{mn})^{24}}.
\end{equation}

 Consequently, \eqref{eq:drp-laurent} has a convergent
Laurent expansion
\[
g(\sigma\tau)
=
\sum_{n=h-m}^{\infty}a_\sigma(n)q_h^n,
\qquad
a_\sigma(h-m)=\nu^{12}\zeta^{-1}\ne0.
\]
Because $h-m\in\mathbb Z$, only finitely many negative
powers occur. This is precisely condition~(2) of
Definition~\ref{def:modular-function}.
Since $\sigma\in \Gamma$ was arbitrary, the condition holds
at every cusp. Hence, we conclude that $ g(\tau)$ is a modular function for $ \Gamma_\alpha$. 
 
\end{proof}

For $a\geq1$ and $0\leq r<a$, define
\begin{equation}\label{eq:Fr}
 F_r(q)=\sum_{\substack{n\geq0\\n\equiv r\; (\mathrm{mod}\;a)}}
             p(n)q^{24n-1}\in\Z((q)).
\end{equation}
The normalization follows from
\begin{equation}\label{eq:f}
 f(\tau):=\eta(24\tau)^{-1}
          =q^{-1}P(q^{24})
          =\sum_{n\geq0}p(n)q^{24n-1}.
\end{equation}

Set $M=24a$ and $\zeta=e^{2\pi i/M}$. For every integer $b$,
\begin{equation}\label{eq:algsec-orthogonality}
 \frac{1}{M}\sum_{t=0}^{M-1}\zeta^{tb}
 =
 \begin{cases}
  1,&M\mid b,\\
  0,&M\nmid b.
 \end{cases}
\end{equation}

Since replacing $\tau$ by $\tau+t/M$ replaces $q$ by $\zeta^tq$,
equation~\eqref{eq:f} yields
\[
 f\left(\tau+\frac{t}{M}\right)
 =\sum_{n\geq0}p(n)\zeta^{t(24n-1)}q^{24n-1}.
\]
Multiply this identity by $\zeta^{-t(24r-1)}$, average over $t$, and
interchange the finite sum with the convergent series. We obtain
\begin{align*}
 \frac{1}{M}\sum_{t=0}^{M-1}
   \zeta^{-t(24r-1)}f\left(\tau+\frac{t}{M}\right)
 &=
 \sum_{n\geq0}p(n)q^{24n-1}
 \left(\frac{1}{M}\sum_{t=0}^{M-1}\zeta^{24t(n-r)}\right).
\end{align*}
By~\eqref{eq:algsec-orthogonality}, we arrive at 
\begin{equation}\label{eq:algsec-filter}
 F_r(q)=\frac{1}{M}\sum_{t=0}^{M-1}
   \zeta^{-t(24r-1)}f\left(\tau+\frac{t}{M}\right).
\end{equation}

\begin{theorem}\label{prop:algebraic-sections}
Let $\ell$ be any prime and set
\begin{equation}\label{eq:K}
 K_\ell=\F_\ell\bigl(\red{j}(q),\red{\Delta}(q)\bigr)
 \subseteq \F_\ell((q)).
\end{equation} 
Then for every integer $a\geq1$ and every $0\leq r<a$, $\red{F_r}(q)$ is algebraic over $K_\ell$.
\end{theorem}
\begin{proof}
Fix $a$ and $r$. We first show that an integral Laurent series
$H_r(q)=\Delta(q)F_r(q)^{24}$   is algebraic over $\C(\red{j}(q))$. We then reduce a polynomial
relation for $H_r$ modulo $\ell$ and obtain a nonzero polynomial over
$K_\ell$ with root $\red{F_r}(q)$.

 For $0\leq t<M$, define
\begin{equation}\label{eq:algsec-eta-ratio}
 u_t(\tau)
 :=\eta(\tau)f\left(\tau+\frac{t}{M}\right)
 =\frac{\eta(\tau)}{\eta(24\tau+t/a)}.
\end{equation}
Here $24t/M=t/a$. Since $\eta^{24}(\tau)=\Delta(\tau)$, we have
\begin{equation}\label{eq:algsec-discriminant-ratio}
 u_t(\tau)^{24}
 =\frac{\Delta(\tau)}{\Delta(24\tau+t/a)}.
\end{equation}
By Theorem \ref{prop:delta-ratio}, 
 $ u_t(\tau)^{24}$ is algebraic over $\C(j(\tau))$, and so   $u_t(\tau)$ is also algebraic over $\C(j(\tau))$.  
 
 Set $L=\mathbb C(j(\tau))$ and regard all functions as
elements of the field $\mathcal M(\mathbb H)$ of meromorphic
functions on the upper half-plane.
Since each $u_t$ is algebraic over $L$,
Lemma~\ref{lem:finite-algebraic-adjunction} shows that
\[
E=L(u_0,\ldots,u_{M-1})
\]
is a finite extension of $L$ and that every element of $E$
is algebraic over $L$.

The identity
\[
\eta(\tau)F_r(q)
=
\frac{1}{M}\sum_{t=0}^{M-1}
\zeta^{-t(24r-1)}u_t(\tau)
\]
shows that $\eta(\tau)F_r(q)\in E$, because all the
coefficients in this finite sum belong to
$\mathbb C\subseteq L$.
Since $E$ is a field, it also contains
\begin{equation}\label{defi:Hr}
H_r(q)
:=
\bigl(\eta(\tau)F_r(q)\bigr)^{24}
=
\Delta(q)F_r(q)^{24}.
\end{equation}
Therefore $H_r(q)$, viewed as a meromorphic function of
$\tau$, is algebraic over $\mathbb C(j(\tau))$.
   
   Note that 
\[
 F_r(q)\in q^{-1}\Z[[q]],
 \qquad
 \Delta(q)=q\prod_{m\geq1}(1-q^m)^{24}\in q\Z[[q]],
\]
and therefore
\begin{equation}\label{eq:algsec-integrality}
 H_r(q)\in q^{-23}\Z[[q]]\subseteq\Z((q)).
\end{equation}

By Lemma~\ref{lem:descent}, $\red{H_r}(q)$ is algebraic over $\mathbb F_\ell(\overline{j}(q))$, that is there exists 
a nonzero polynomial $Q(V)=b_sV^s+\cdots +b_1V+b_0$, where  $b_i \in \mathbb F_\ell(\overline{j}(q))$ such that  
\begin{equation}\label{eq:algsec-reduced-relation}
 Q\bigl(\red{H_r}(q)\bigr)=0.
\end{equation}

Coefficientwise reduction is a ring homomorphism on $\Z((q))$, thus, 
\eqref{defi:Hr} reduces to
\begin{equation}\label{eq:algsec-reduced-power}
 \red{H_r}(q)
 =\red{\Delta}(q)\,\red{F_r}(q)^{24}.
\end{equation}
  The coefficient of $q$ in $\Delta(q)$ is $1$,
so $\red{\Delta}(q)=q+O(q^2)\neq0$ for every prime $\ell$.

Since $\red{\Delta}(q)\in K_\ell$, the polynomial
\begin{equation}\label{eq:algsec-annihilating-polynomial}
 B(X):=Q(DX^{24})
      =\sum_{i=0}^{d}b_iD^iX^{24i}
      \in K_\ell[X]
\end{equation}
is well defined. Its leading coefficient is $b_dD^d\neq0$,
so it is a nonzero polynomial of degree $24d$. Finally,
\[
 B\bigl(\red{F_r}(q)\bigr)
 =Q\bigl(D\,\red{F_r}(q)^{24}\bigr)
 =Q\bigl(\red{H_r}(q)\bigr)=0.
\]
This is precisely the assertion that $\red{F_r}(q)$ is algebraic
over $K_\ell$.
\end{proof}

Applying Theorem \ref{prop:algebraic-sections} and Voloch's theorem, we derive the following result:

\begin{theorem}\label{cor:transcendence}
For every prime $\ell$, $q$ is transcendental over $K_\ell$ defined
in~\eqref{eq:K}. For any fixed $a\geq1$, it remains transcendental over
\begin{equation}\label{eq:L}
 L_{\ell,a}:=
 K_\ell\bigl(\red{F_0}(q),\ldots,\red{F_{a-1}}(q)\bigr).
\end{equation}
\end{theorem}

\begin{proof}
Fix a prime $\ell$ and let $k=\overline{\F}_\ell$. Write
$\mathcal T_\ell=k\bigl(\red{A_4}(q),\red{A_6}(q)\bigr).$ 
By Theorem~\ref{thm:voloch}, it suffices to show that  
\begin{equation}\label{eq:transcor-inclusion}
 K_\ell\subseteq\mathcal T_\ell.
\end{equation}

Coefficientwise reduction modulo $\ell$ of \eqref{eq:invariant-delta}
 and \eqref{eq:invariant-j}
 preserves sums and products.
Consequently,  
\begin{align*}
 \red{\Delta}(q)
 &=(\red{A_4}(q))^2-\red{A_6}(q)-64(\red{A_4}(q))^3
   -432(\red{A_6}(q))^2+72\red{A_4}(q)\red{A_6}(q), \\[5pt]
 \red{\Delta}(q)\red{j}(q)&= {(1-48\red{A_4}(q))^3}.
\end{align*}
where the integer coefficients on the right are interpreted in
$\F_\ell$. The first identity shows that $\red{\Delta}(q)\in\mathcal T_\ell$.
Moreover,   the coefficient of $q$ in $\Delta(q)$ is $1$,
so $\red{\Delta}(q)=q+O(q^2)\neq0$ for every prime $\ell$. Since $\mathcal T_\ell$ is a field, it
contains $(\red{\Delta}(q))^{-1}$, and hence
\[
\red{j}(q)=\frac{(1-48\red{A_4}(q))^3}{\red{\Delta}(q)}
      \in\mathcal T_\ell.
\]
Together with $\F_\ell\subseteq k$, this proves
\[
 K_\ell=\F_\ell\bigl(\red{j}(q),\red{\Delta}(q)\bigr)
      \in\mathcal T_\ell= k\bigl(\red{A_4}(q),\red{A_6}(q)\bigr).
\]
These calculations apply also when $\ell=2$ or $3$. We have reduced
identities with integer coefficients and divided only by the
nonzero series $\red{\Delta}(q)$; no division by an integer divisible by
$\ell$ is involved.

Suppose now that $q$ were algebraic over $K_\ell$. There would
then be a nonzero polynomial $B(T)\in K_\ell[T]$ with $B(q)=0$.
By~\eqref{eq:transcor-inclusion}, the same polynomial belongs to
$\mathcal T_\ell[T]$. It remains nonzero because the inclusion
of fields is injective. This contradicts the transcendence of
$q$ over $\mathcal T_\ell$. Thus $q$ is transcendental over
$K_\ell$.

 Now fix $a\geq1$ and put
\[
L_{\ell,a}
=
K_\ell\bigl(
\overline{F_0}(q),\ldots,\overline{F_{a-1}}(q)
\bigr).
\]
By Theorem~\ref{prop:algebraic-sections}, each
$\overline{F_i}(q)$ is algebraic over $K_\ell$.
Since only finitely many elements are adjoined,
Lemma~\ref{lem:finite-algebraic-adjunction} implies that
$L_{\ell,a}/K_\ell$ is a finite extension.

We now  show that $q$ is transcendental  over
$L_{\ell,a}$.
Suppose, to the contrary, that $q$ is algebraic over
$L_{\ell,a}$.
Then
\[
[L_{\ell,a}(q):L_{\ell,a}]<\infty.
\]
Together with $[L_{\ell,a}:K_\ell]<\infty$, the tower law
gives
\[
[L_{\ell,a}(q):K_\ell]
=
[L_{\ell,a}(q):L_{\ell,a}]
[L_{\ell,a}:K_\ell]
<\infty.
\]
Every element of a finite extension is algebraic over
the base field.
Thus $q$ would be algebraic over $K_\ell$, contradicting
the transcendence established above.
Therefore $q$ remains transcendental over $L_{\ell,a}$.

\end{proof}

\section{Proof of the nonexistence theorem}\label{sec:proof}

For $a\geq1$, put
\begin{equation}\label{eq:block-gf}
 B_a(x)=(x+x^2+\cdots+x^a)P(x)
        =\sum_{N\geq0}b_a(N)x^N.
\end{equation}
Coefficient extraction gives
\begin{equation}\label{eq:block}
 b_a(N)=\sum_{j=1}^{a}p(N-j)
       =c(N)-c(N-a)\qquad(N\geq a).
\end{equation}

For a series $A(x)=\sum_{N\geq0}A_Nx^N$ and $0\leq b<a$, define its
arithmetic section by
\begin{equation}\label{eq:section}
 \mathcal S_{a,b}A(x)
 :=\sum_{\substack{N\geq0\\N\equiv b\; (\mathrm{mod}\;a)}}A_Nx^N.
\end{equation}
This operator keeps the original exponents; it does not replace $an+b$ by $n$.

\begin{lemma}\label{lem:block-polynomial}
Let $\ell$ be a prime and $0\leq b<a$. If there  exists integers
$d$ and $n_0\geq0$ such that $ c(an+b)\equiv d\pmod \ell$ for every $n\geq n_0$, then there exists $R(x)\in\F_\ell[x]$
\begin{equation}\label{eq:polynomial-section}
 \mathcal S_{a,b}{B_a}(x)\equiv R(x) \pmod{\ell }.
 \end{equation}
 
\end{lemma}

\begin{proof}
For all sufficiently large $n$, both $c(an+b)$ and $c(a(n-1)+b)$ have the
same residue. Equation~\eqref{eq:block} therefore gives
$b_a(an+b)\equiv0\pmod\ell$. Only finitely many coefficients of the section
can be nonzero, so the section is a polynomial.
\end{proof}

\begin{proof}[Proof of Theorem~\ref{thm:main}]
Suppose~\eqref{eq:forbidden} holds. Choose a prime $\ell\mid m$ and reduce
the congruence modulo $\ell$. Write $b=at+b_0$, with $0\leq b_0<a$.
Replacing $b$ by $b_0$ only changes the initial point of the progression,
so the resulting sequence is still eventually constant. We may therefore
assume $0\leq b<a$.

Write
\[
 \overline P(x)
 :=\sum_{n\geq0}\overline{p(n)}x^n
 \in\mathbb F_\ell[[x]].
\]
By \eqref{eq:R} and the linearity of the arithmetic
section operator, we have
\[
S_{a,b}\overline B_a(x)
 =\mathcal S_{a,b}
   \left(\left(\sum_{j=1}^{a}x^j\right)\overline P(x)\right)
 =\sum_{j=1}^{a}
   \mathcal S_{a,b}\bigl(x^j\overline P(x)\bigr).
\]

Fix $j\in\{1,\ldots,a\}$. The definition of $\mathcal S_{a,b}$ gives
\[
 \mathcal S_{a,b}\bigl(x^j\overline P(x)\bigr)
 =
 \sum_{\substack{N\geq j\\N\equiv b\;(\mathrm{mod}\;a)}}
 \overline{p(N-j)}x^N.
\]
Now set $n=N-j$. Then $N\geq j$ is equivalent to
$n\geq0$, and
\[
 N\equiv b\pmod a
 \quad\Longleftrightarrow\quad
 n+j\equiv b\pmod a
 \quad\Longleftrightarrow\quad
 n\equiv r_j\pmod a.
\]
Consequently,
\[
 \mathcal S_{a,b}\bigl(x^j\overline P(x)\bigr)
 =
 \sum_{\substack{n\geq0\\n\equiv r_j\;(\mathrm{mod}\;a)}}
 \overline{p(n)}x^{n+j}.
\]
Summing over $j$, we obtain
\begin{equation}\label{eq:S}
S_{a,b} \overline B_a(x)
 =
 \sum_{j=1}^{a}
 \sum_{\substack{n\geq0\\n\equiv r_j\;(\mathrm{mod}\;a)}}
 \overline{p(n)}x^{n+j}.
\end{equation}

By Lemma~\ref{lem:block-polynomial},
\begin{equation}\label{eq:R}
 \mathcal S_{a,b}\red B_a(x)=R(x)
 \qquad\text{for some }R(x)\in\F_\ell[x].
\end{equation}
 Substituting $x=q^{24}$
in~\eqref{eq:R} and applying \eqref{eq:S} consequently gives
\begin{align}
 R(q^{24})
 &=\sum_{j=1}^{a}
       \sum_{\substack{n\geq0\\n\equiv r_j\; (\mathrm{mod}\;a)}}
       \red{p(n)}q^{24(n+j)}\notag\\
       &=\sum_{j=1}^{a}q^{24j+1}
       \sum_{\substack{n\geq0\\n\equiv r_j\; (\mathrm{mod}\;a)}}
       \red{p(n)}q^{24n-1}\notag\\
 &=\sum_{j=1}^{a}q^{24j+1}\red F_{r_j}(q).
 \label{eq:contradiction-series}
\end{align}
where $\red F_{r_j}(q)$ is defined in \eqref{eq:Fr}. 
 
 Let $L=L_{\ell,a}$ be the field defined in \eqref{eq:L},
and let $X$ be an indeterminate over $L$.
Since $\red{F_{r_j}}(q)\in L$ for every $j$ and
$R(x)\in\mathbb F_\ell[x]$, the expression
\begin{equation}\label{eq:Q}
 Q(X)
 :=\sum_{j=1}^{a}\red{F_{r_j}}(q)X^{24j+1}
   -R(X^{24})
\end{equation}
is a polynomial in $L[X]$.
Here the series $\red{F_{r_j}}(q)$ are regarded as
coefficients, while $X$ is the polynomial variable.
Substituting $X=q$ and using
\eqref{eq:contradiction-series}, we obtain
\[
 Q(q)
 =\sum_{j=1}^{a}q^{24j+1}\red{F_{r_j}}(q)
  -R(q^{24})
 =0.
\]

We now show that $Q(X)$ is not the zero polynomial.
Evidently,  for each $j$, the coefficient of
$X^{24j+1}$ in $Q(X)$ is exactly
$\red{F_{r_j}}(q)$: no other term in \eqref{eq:Q}
contributes to this coefficient.

It therefore suffices to show that at least one of
the series $\red{F_{r_j}}(q)$ is nonzero.
Since the integers $1,\ldots,a$ represent all residue
classes modulo $a$ exactly once, there is a unique
$j^*\in\{1,\ldots,a\}$ such that
$ j^*\equiv b\pmod a.$ 
By the definition of $r_j$, we have
$ r_{j^*}\equiv b-j^*\equiv0\pmod a.
$ 
As $0\leq r_{j^*}<a$, it follows that $r_{j^*}=0$.

By the definition of $F_0(q)$,
\begin{equation}\label{eq:nonzero}
 \red{F_0}(q)
 =\sum_{m\geq0}\red{p(am)}q^{24am-1}
 =q^{-1}
  +\sum_{m\geq1}\red{p(am)}q^{24am-1}.
\end{equation}
The coefficient of $q^{-1}$ is
$\red{p(0)}=\red{1}=1$, which is nonzero in
$\mathbb F_\ell$ for every prime $\ell$.
All terms with $m\geq1$ have exponent strictly greater
than $-1$, so none can cancel this term.
Thus $\red{F_0}(q)\neq0$.

Therefore the coefficient of $X^{24j^*+1}$ in $Q(X)$
is the nonzero element $\red{F_0}(q)$.
This proves that $Q(X)\neq0$.
Together with $Q(q)=0$, this contradicts the
transcendence of $q$ over $L$.

We have obtained a nonzero polynomial in $L[X]$ with root $q$, contradicting
Theorem~\ref{cor:transcendence}. This proves the theorem.
\end{proof}

 \vskip 0.2cm
\noindent{\bf Acknowledgment.} This work
was supported by   the National Science Foundation of China.  

 \vskip 0.2cm
 \noindent{\bf Declaration of AI usage.}
During the preparation of this work, the author used AI
to explore proof strategies and to assist with language
editing and proofreading. The author takes full responsibility
for the final content of the publication.

\end{document}